\documentclass[11pt,a4paper]{article}

\usepackage[utf8]{inputenc}
\usepackage[T1]{fontenc}
\usepackage{amsmath,amssymb,amsthm}
\usepackage{mathtools}
\usepackage{hyperref}
\hypersetup{colorlinks=true, linkcolor=black, citecolor=black, urlcolor=black}
\usepackage{geometry}
\theoremstyle{plain}
\newtheorem{theorem}{Theorem}[section]
\newtheorem{lemma}{Lemma}[section]
\newtheorem{corollary}{Corollary}[section]
\theoremstyle{remark}
\newtheorem{remark}{Remark}[section]

\newcommand{\HH}{\mathbb H}
\newcommand{\RR}{\mathbb R}
\newcommand{\V}{V}
\newcommand{\B}{\mathcal B}
\newcommand{\Bo}{\mathcal B_0}
\newcommand{\F}{\mathcal F}
\newcommand{\W}{\mathcal W}
\newcommand{\G}{\mathcal G}
\newcommand{\Br}{\mathfrak{Br}}
\newcommand{\Boo}{\mathcal B_{00}}
\newcommand{\CA}{\mathcal C_A}
\newcommand{\sS}{\sigma_S}
\newcommand{\sigb}{\sigma_b^S}
\newcommand{\sd}{\sigma_d^S}
\newcommand{\rS}{\rho_S}

\DeclareMathOperator{\ran}{ran}
\DeclareMathOperator{\asc}{asc}
\DeclareMathOperator{\dsc}{dsc}
\DeclareMathOperator{\ind}{ind}
\DeclareMathOperator{\Id}{I}
\DeclareMathOperator{\codim}{codim}

\newcommand{\Qq}[1]{Q_q(#1)}

\title{A note on the Browder S-spectrum\\
of a bounded quaternionic operator}
\author{
   Mohamed Ali Dbeibia\\[0.5em]
  \small University of Sfax, Route de Soukra Km 3.5, B.P.\ 1171, 3000 Sfax, Tunisia.\\
  \small \texttt{medalidbeibia@gmail.com}
}
\date{}

\begin{document}
\maketitle

\begin{abstract}
Let $A$ be a bounded operator on a separable right quaternionic Hilbert space and let $\CA$ be the set of compact operators commuting with $A$. We establish the formula
\[
\sigb(A)=\bigcap_{K\in\CA} \sS(A+K),
\]
a quaternionic analogue of the classical characterization of the complex Browder spectrum. The proof rests on a factorization lemma (Lemma~\ref{lem:fact}), a consequence of the Riesz--Schauder theorem and an ascent/descent argument, and on the characterization $\sigb(A)=\sS(A)\setminus\sd(A)$ in terms of S-eigenvalues of finite type, due to Arzini and Jaatit \cite{AJ26}, for which we give here an independent proof (Lemma~\ref{lem:A}). We also record two consequences of independent interest: the intersection defining $\sigb(A)$ may be restricted, without loss, to \emph{finite-rank} operators commuting with $A$ (Corollary~\ref{cor:finite}); and, by a separate argument resting on the ascent/descent stability results of \cite{KD24}, the stronger pointwise statement $\sigb(A+K)=\sigb(A)$ holds for a fixed $K\in\CA$ and arbitrary $A\in\B(\V)$ (Theorem~\ref{thm:pointwise}), known in the complex case.
\end{abstract}

\medskip
\noindent\textbf{Keywords:} Quaternionic operator; S-spectrum; Browder S-spectrum; Riesz idempotent; compact perturbation; ascent and descent.

\noindent\textbf{2020 Mathematics Subject Classification:} Primary 47A53; Secondary 47A10, 47A55, 47B07, 46S05.

\section{Introduction}

In the complex case, two classical results describe the parts of the spectrum of a bounded operator $T$ on a Hilbert space that are resistant to compact perturbations. The Weyl spectrum is the intersection of the spectra perturbed by an arbitrary compact operator, while the Browder spectrum $\sigma_b(T)$ is the intersection of the spectra perturbed by a \emph{commuting} compact operator:
\[
\sigma_b(T)=\bigcap\big\{\sigma(T+K)\ :\ K\ \text{compact},\ TK=KT\big\}.
\]
This formula comes with a stronger pointwise statement, due to Kaashoek and Lay \cite{KL72}; see also Grabiner \cite{Gra78}: finite ascent and descent are preserved under commuting Riesz perturbations, whence $\sigma_b(T+K)=\sigma_b(T)$ for every compact $K$ commuting with $T$. The same stability phenomenon extends, more generally, to the single-valued extension property under commuting Riesz perturbations, as shown by Aiena and M\"uller \cite{AM15}. We refer to the monographs of Aiena \cite{Ai04,Ai18} for a systematic treatment of Browder and Weyl theory, ascent and descent, and their interplay with local spectral theory in the classical (complex) setting.

It is worth recalling briefly where the notion of S-spectrum itself comes from. The original motivation was quantum-mechanical: Birkhoff and von Neumann \cite{BvN36} showed that quantum mechanics can equally well be formulated over the quaternions, without specifying which notion of spectrum a quaternionic linear operator should carry, and the question of the correct spectral notion remained open for decades, the naive left and right spectra being unsuited to a genuine spectral theory. It was only in 2006 that Colombo and Sabadini identified the S-spectrum, by purely hyperholomorphic methods, building on the theory of slice regular functions introduced shortly before by Gentili and Struppa \cite{GS07}. The corresponding S-functional calculus, a quaternionic counterpart of the Riesz--Dunford calculus, was first announced in \cite{CGSS07}, then developed for bounded operators in \cite{CS09,CGSS10}, and extended to $n$-tuples of not necessarily commuting operators in \cite{CSS08}; this body of work was later collected in the monograph \cite{CSS11}. The Hilbert space theory proper, including the quaternionic spectral theorem for normal operators formulated on the S-spectrum, was developed shortly afterwards by Alpay, Colombo and Kimsey \cite{ACK16}. We refer to the monograph of Colombo, Gantner and Kimsey \cite{CGK} for a systematic and comprehensive account of the whole theory, including the Fredholm theory used throughout the present note.

In the quaternionic setting, the relevant notion of spectrum is thus the S-spectrum, associated with the pseudo-resolvent $\Qq{A}=A^2-2\operatorname{Re}(q)A+|q|^2\Id$. Muraleetharan and Thirulogasanthar \cite{MT18,MT19} developed Fredholm theory and introduced the Weyl and Browder S-spectra in this setting. They \emph{define} the Weyl S-spectrum as the intersection $\bigcap_{K\in\Bo(\V)}\sS(A+K)$ (\cite[Def.~6.1]{MT19}) and identify it with the set of $q$ for which $\Qq{A}$ is not a Weyl operator (\cite[Thm~6.6]{MT19}). For the Browder S-spectrum, however, they proceed differently: they define it via the condition of being Browder on the pseudo-resolvent (\cite[Def.~7.12]{MT19}) and explicitly note, at the start of their Section~7 and again in the conclusion, that they do not have the Riesz idempotent on the whole of $\HH$, and leave the corresponding study to future work. Their Theorem~7.10 nonetheless gives a partial implication: if $B$ is Browder and $0\in\sS(B)$, then $0$ is an isolated point of $\sS(B)$ and $B^2$ is a Weyl operator.

The quaternionic Riesz idempotent, introduced by Colombo and collaborators \cite{CSS11,CGK} within the S-functional calculus, and the notion of S-eigenvalue of finite type, have since been developed by Baloudi, Belgacem and Jeribi \cite{BBJ22}, and then by Baloudi, Jeribi and Zmouli \cite{BJZ24}, which makes the intersection formula for $\sigb$ accessible. This is the subject of the present note.

Two recent works directly overlap with this subject and should be pointed out from the outset. Arzini and Jaatit examine isolated parts of the S-spectrum and their associated spectral projectors \cite{AJ25}, and then characterize Riesz points of the S-spectrum and establish a quaternionic analogue of Weyl's theorem \cite{AJ26}. Their set $\pi_0(T)$ of Riesz points coincides with our $\sd(A)$, and their Theorem~4.2 states that, for $q\in\sS(A)$, $q$ is a Riesz point if and only if $\Qq{A}$ is Fredholm with finite ascent and descent: this is exactly Lemma~\ref{lem:A} below. This lemma is therefore not a new result, and we attribute it to \cite{AJ26}; we nonetheless give here an independent, self-contained proof of it. In addition, Kumari and Dharmarha \cite{KD24} study quaternionic Browder operators: they show that the ascent and descent of an operator are stable under a compact perturbation commuting with it (\cite[Thm~3.9]{KD24}, the quaternionic counterpart of the ascent/descent stability theorem of Kaashoek and Lay \cite{KL72}), that a product of commuting Browder operators is Browder (\cite[Prop.~4.1]{KD24}), that the Browder property splits along an invariant direct-sum decomposition (\cite[Thm~4.4]{KD24}), that a spectral mapping theorem $\sigb(p(T))=p(\sigb(T))$ holds for self-adjoint $T$ and real $p$ (\cite[Thm~4.8]{KD24}), and that the inclusion $\sigb(T+A)\setminus\{0\}\subseteq[\sigb(T)\cup\sigb(A)]\setminus\{0\}$ holds for commuting $T,A$ (\cite[Thm~4.9]{KD24}).

The pointwise stability $\sigb(A+K)=\sigb(A)$ for arbitrary $A\in\B(\V)$ and fixed $K\in\CA$ follows from \cite[Thm~3.9]{KD24} by a symmetry argument: applying the theorem to $\Qq{A}$ with perturbation $K'$ gives one inclusion, and applying it to $\Qq{A+K}$ with perturbation $-K'$ gives the reverse, using the stability of the Fredholm property under compact perturbations. We record this in Theorem~\ref{thm:pointwise} below; \cite{KD24} itself states this spectral consequence only in the narrower \cite[Cor.~3.10]{KD24}, under the sufficient but not necessary extra hypothesis that $A$ already be Browder.

By contrast, none of \cite{AJ25,AJ26,KD24} contains the aggregation over the commutant ideal $\CA$ of the kind appearing in Theorem~\ref{thm:main}: \cite{AJ26} only involves the Weyl S-spectrum $\bigcap_{K\in\Bo(\V)}\sS(A+K)$, borrowed from \cite{MT19}; and \cite{KD24} is confined throughout to a single, fixed compact perturbation. Theorem~\ref{thm:main}, the factorization lemma underlying it, and their corollaries are therefore, to our knowledge, new.

The note is organized around a single statement, the \emph{factorization lemma} (Lemma~\ref{lem:fact}): if $K$ is compact and commutes with $A$, and if $\Qq{A+K}$ is invertible, then $\Qq{A}$ is Browder. The argument is the one, classical in the complex case, of the factorization $\Qq{A}=R(\Id-R^{-1}K')$ with $R=\Qq{A+K}$ invertible and $K'$ compact; what we verify is that it carries over unchanged to the quaternionic setting, the non-commutativity of the scalars never coming into play. Its contrapositive gives the hard inclusion of the theorem (Corollary~\ref{cor:stab}); applied to the Riesz projector $K=P_{[q]}$, it also gives the easy half of the characterization $\sigb(A)=\sS(A)\setminus\sd(A)$. It rests on the quaternionic analogue of the Riesz--Schauder theorem (Lemma~\ref{lem:RS}), for which we give a proof so that the note is self-contained; the corresponding ascent/descent stability result is given in \cite[Thm~3.9]{KD24}.

One last point of positioning deserves to be made explicit: the second half of Lemma~\ref{lem:A} goes back to Theorem~7.10 of \cite{MT19}. The decomposition $\ran(B^m)\oplus\ker(B^m)$ and the conclusion by disconnection of the S-spectrum are already present there; what we add is the transfer of the statement, formulated in \cite{MT19} at the point $0$ for the operator $B$, to the point $q$ for the operator $A$, together with an elementary localization argument (Lemma~\ref{lem:loc}) that replaces the appeal to the spectral mapping theorem with a Bézout identity in $\RR[x]$. The same pattern is found in \cite[Thm~3.15]{AJ26}.

This note is organized as follows. Section~\ref{sec:prelim} gathers the preliminaries and notation on the S-spectrum, ascent and descent, and the Browder property, together with the properties (R1)--(R8) taken from the literature and used throughout. Section~\ref{sec:fact} establishes the factorization lemma (Lemma~\ref{lem:fact}), by way of the auxiliary Lemmas~\ref{lem:invar} and~\ref{lem:RS}, and derives from it the one-sided stability statement of Corollary~\ref{cor:stab}. Section~\ref{sec:charac} is devoted to the characterization of the Browder S-spectrum: it proves $\sigb(A)=\sS(A)\setminus\sd(A)$ (Lemma~\ref{lem:A}, due to \cite{AJ26}, for which we give here an independent, self-contained proof), and, building on it, establishes the main theorem (Theorem~\ref{thm:main}), together with the pointwise stability theorem (Theorem~\ref{thm:pointwise}) and the finite-rank refinement (Corollary~\ref{cor:finite}). Section~\ref{sec:conclusion} concludes.

\section{Preliminaries and notation}\label{sec:prelim}

Let $\V$ be a nonzero separable right quaternionic Hilbert space, equipped with a left multiplication fixed by a Hilbert basis, so that $\B(\V)$ is a two-sided quaternionic Banach $C^*$-algebra (\cite[Thm~5.1]{MT19}). We denote by $\Bo(\V)$ the set of compact operators, by $\Boo(\V)$ the set of finite-rank operators, by $\F(\V)$ the set of Fredholm operators, by $\W(\V)=\{B\in\F(\V):\ind(B)=0\}$ the set of Weyl operators, and by $\G(\V)$ the set of invertible operators.

For $A\in\B(\V)$ and $q\in\HH$, we define the pseudo-resolvent
\[
\Qq{A}=A^2-2\operatorname{Re}(q)A+|q|^2\,\Id,
\]
and the S-spectrum $\sS(A)=\{q\in\HH:\Qq{A}\text{ not invertible}\}$, whose complement is the S-resolvent set $\rS(A)$. We denote by
\[
[q]=\{x\in\HH:\operatorname{Re}(x)=\operatorname{Re}(q),\ |x|=|q|\}
\]
the $2$-sphere of $q$; when $q\in\RR$, $[q]$ reduces to the singleton $\{q\}$. Finally $Q_q(x)=x^2-2\operatorname{Re}(q)x+|q|^2$ also denotes the corresponding real polynomial in $\RR[x]$, so that $\Qq{A}$ is precisely $Q_q$ evaluated at the operator $A$, that is, a real polynomial in $A$ with central coefficients; consequently, for each fixed $q$, $\Qq{A}$ commutes with every operator that commutes with $A$, a fact we invoke repeatedly below.

Let us recall the following definitions, taken from \cite{MT19}.
\begin{itemize}
\item The ascent $\asc(B)$ and descent $\dsc(B)$ of an operator $B\in\B(\V)$ are the smallest integers $n\ge0$ such that $\ker(B^{n+1})=\ker(B^n)$ and $\ran(B^{n+1})=\ran(B^n)$ respectively, with the convention $+\infty$ if they do not exist (\cite[Def.~7.2]{MT19}).
\item An operator $B$ is said to be \emph{Browder}, denoted $B\in\Br(\V)$, if it is Fredholm and $\asc(B)<\infty$ and $\dsc(B)<\infty$ (\cite[Def.~7.5]{MT19}):
\[
\Br(\V)=\{B\in\F(\V):\asc(B)<\infty\ \text{and}\ \dsc(B)<\infty\}.
\]
By \cite[Lemma~7.3(b)]{MT19} we then have $\asc(B)=\dsc(B)$. The Fredholm condition does not follow from the other two: the zero operator on an infinite-dimensional space satisfies $\asc=\dsc=1$ without being Fredholm.
\item The Browder S-spectrum is $\sigb(A)=\{q\in\HH:\Qq{A}\notin\Br(\V)\}$ (\cite[Def.~7.12]{MT19}).
\item We denote by $\sd(A)$ the set of \emph{S-eigenvalues of finite type} of $A$: those $q\in\sS(A)$ such that $[q]$ is an isolated sphere of $\sS(A)$ whose associated Riesz projector $P_{[q]}$ has finite rank (\cite[Def.~3.2, Thm~3.5]{BJZ24}, \cite[Thm~3.10]{BBJ22}).
\end{itemize}

We shall use the following results, all established in the cited literature.

\medskip
\noindent\textbf{(R1)} \cite[Lemma~7.7]{MT19}: if $\asc(B)=\dsc(B)=m<\infty$, then $\V=\ran(B^m)\oplus\ker(B^m)$ (algebraic direct sum, $B$-invariant subspaces).

\medskip
\noindent\textbf{(R2)} \cite[Prop.~3.6, 4.2, Thm~4.5, 4.9, 4.11, Cor.~4.13]{MT19}: every finite-rank operator is compact; every Fredholm operator has closed range and finite-dimensional kernel; if $A$ is bijective and $K$ is compact, $A+K$ is Fredholm; a product of Fredholm operators is Fredholm, with additive index; the index is stable under compact perturbation; if $A\in\F(\V)$, then $A^n\in\F(\V)$ and $\ind(A^n)=n\ind(A)$.

\medskip
\noindent\textbf{(R3)} \cite[Rem.~7.9]{MT19}: $\G(\V)\subset\Br(\V)\subset\W(\V)\subset\F(\V)$.

\medskip
\noindent\textbf{(R4)} \cite[Prop.~3.16]{MT19}: for every $A\in\B(\V)$ acting on a nonzero space, $\sS(A)$ is a nonempty compact set; on the zero space, $\sS(A)=\emptyset$. In all cases $\sS(A)$ is closed and axially symmetric, that is, a union of spheres $[q]$: indeed $\Qq{A}$ depends only on $\operatorname{Re}(q)$ and $|q|$, hence only on $[q]$. These last two properties require no assumption on the dimension and hold trivially on the zero space.

\medskip
\noindent\textbf{(R5)} \cite[Thm~2.12]{BJZ24}, after \cite[Thm~4.1.5]{CGK}: if $[q]$ is an isolated sphere of $\sS(A)$, the associated Riesz projector $P_{[q]}$ is a bounded projector commuting with $A$; in particular $\ran(P_{[q]})$ and $\ker(P_{[q]})$ are closed and $A$-invariant.

\medskip
\noindent\textbf{(R6)} \cite[Thm~3.3]{BJZ24}, after \cite[Thm~4.4]{ACS14}: for every bounded projector $P\in\B(\V)$ commuting with $A$,
\[
\sS(A)=\sS(A|_{\ran P})\cup\sS(A|_{\ker P}).
\]
This union is not disjoint in general: for $A=\Id$ and arbitrary $P$, both terms equal $\{1\}$. Disjointness will need to be established separately at each use.

\medskip
\noindent\textbf{(R7)} \cite[Prop.~3.7]{BJZ24}: if $[q]$ is an isolated sphere of $\sS(A)$, with Riesz projector $P_{[q]}$ (see \cite{CGK,CSS11} for more details on its construction), then
\[
\Qq{A}+\big[2A+(1-2\operatorname{Re}(q))\Id\big]P_{[q]}
\]
is invertible.

\medskip
\noindent\textbf{(R8)} \cite[Prop.~3.8]{BJZ24}: if $[q]$ is an isolated sphere of $\sS(A)$, then $q\in\sd(A)$ if and only if $\Qq{A}\in\W(\V)$.

\begin{remark}[degenerate sphere case]\label{rem:reel}
Statements (R5), (R7) and (R8) are formulated in \cite{BJZ24} for isolated $2$-spheres, the motivation there being the case $q\in\sS(A)\setminus\RR$. They also hold when $q\in\RR$, where $[q]=\{q\}$, each for its own reason. For (R5), Theorem~4.1.5 of \cite{CGK} only assumes the decomposition $\sS(A)=\sigma_1\cup\sigma_2$ with $\operatorname{dist}(\sigma_1,\sigma_2)>0$ and $\sigma_1$ axially symmetric, which an isolated real singleton satisfies. For (R7), the proof of \cite[Prop.~3.7]{BJZ24} uses, besides the isolation of $[q]$, only the fact that the function $p\mapsto 1-2\operatorname{Re}(q)+2p$ does not vanish on $[q]$; for real $q$ it is there constantly equal to $1$. For (R8), the proof of \cite[Prop.~3.8]{BJZ24} follows from (R7) and an argument in the quaternionic Calkin algebra, indifferent to whether $q$ is real or not.
\end{remark}

Finally, the following elementary identity will be used twice.

\begin{remark}[perturbation identity]\label{rem:pert}
Let $A,K\in\B(\V)$ with $AK=KA$. Then $(A+K)^2=A^2+2AK+K^2$, whence, for every $q\in\HH$,
\begin{equation}\label{eq:pert}
\Qq{A+K}=\Qq{A}+2AK+K^2-2\operatorname{Re}(q)K .
\end{equation}
If moreover $K$ is a projector, $K^2=K$ and the correction term simplifies to $\big[2A+(1-2\operatorname{Re}(q))\Id\big]K$.
\end{remark}

\section{The factorization lemma}\label{sec:fact}

\begin{lemma}[invariance of ascent and descent]\label{lem:invar}
Let $X,R\in\B(\V)$ with $R$ invertible and $RX=XR$. Then, for every $n\ge0$,
\[
\ker\big((RX)^n\big)=\ker(X^n),\qquad \ran\big((RX)^n\big)=\ran(X^n).
\]
In particular $\asc(RX)=\asc(X)$ and $\dsc(RX)=\dsc(X)$.
\end{lemma}

\begin{proof}
By commutativity, $(RX)^n=R^nX^n$. Since $R^n$ is injective, $\ker(R^nX^n)=\ker(X^n)$. Since $RX^n=X^nR$ and $R$ is surjective,
\[
R\big(\ran(X^n)\big)=R\big(X^n(\V)\big)=X^n\big(R(\V)\big)=X^n(\V)=\ran(X^n);
\]
iterating, $R^n(\ran(X^n))=\ran(X^n)$, whence $\ran(R^nX^n)=R^n(\ran(X^n))=\ran(X^n)$.
\end{proof}

\begin{lemma}[Riesz--Schauder]\label{lem:RS}
Let $B_1\in\Bo(\V)$. Then $\Id-B_1\in\Br(\V)$.
\end{lemma}

\begin{proof}
\emph{Fredholm of index $0$.} By (R2) applied to $\Id$ bijective and $-B_1$ compact, $\Id-B_1\in\F(\V)$, and $\ind(\Id-B_1)=\ind(\Id)=0$ by stability of the index under compact perturbation. Again by (R2), $(\Id-B_1)^n\in\F(\V)$ with $\ind((\Id-B_1)^n)=n\cdot0=0$ for every $n\ge0$.

\emph{Finite ascent.} Suppose for contradiction that the chain $\ker((\Id-B_1)^n)$ is strictly increasing for every $n$. These are closed subspaces, being kernels of bounded operators, and finite-dimensional since $(\Id-B_1)^n\in\F(\V)$. For each $n\ge1$, choose $x_n$ of norm $1$ in $\ker((\Id-B_1)^{n+1})$ orthogonal to $\ker((\Id-B_1)^n)$; this is possible since the orthogonal complement of a proper closed subspace of a quaternionic Hilbert space is nonzero. For $m<n$, set
\[
y:=(\Id-B_1)x_n-(\Id-B_1)x_m+x_m.
\]
Then
\[
(\Id-B_1)^n y=(\Id-B_1)^{n+1}x_n-(\Id-B_1)^{n+1}x_m+(\Id-B_1)^n x_m=0:
\]
the first term vanishes because $x_n\in\ker((\Id-B_1)^{n+1})$, and the other two because $m+1\le n$ implies $x_m\in\ker((\Id-B_1)^{m+1})\subseteq\ker((\Id-B_1)^n)$. Hence $y\in\ker((\Id-B_1)^n)$ and, by construction, $x_n\perp y$. A direct computation gives
\[
x_n-y=x_n-(\Id-B_1)x_n+(\Id-B_1)x_m-x_m=B_1x_n-B_1x_m.
\]
The conjugate symmetry of the inner product gives $\langle x_n,y\rangle=\overline{\langle y,x_n\rangle}=0$, and the Pythagorean theorem applies:
\[
\|B_1x_n-B_1x_m\|^2=\|x_n-y\|^2=\|x_n\|^2+\|y\|^2\ge\|x_n\|^2=1.
\]
Thus $\|B_1x_n-B_1x_m\|\ge1$ for all $m<n$: the bounded sequence $(x_n)$ has an image $(B_1x_n)$ with no Cauchy subsequence, contradicting the compactness of $B_1$. The chain of kernels therefore stabilizes at a finite stage $m_0=\asc(\Id-B_1)$.

\emph{Finite descent.} Set $d:=\dim\ker((\Id-B_1)^{m_0})<\infty$; we proceed by a counting argument rather than by dualizing the previous one. For every $n$, $(\Id-B_1)^n$ is Fredholm of index $0$, so
\[
\codim\ran\big((\Id-B_1)^n\big)=\dim\ker\big((\Id-B_1)^n\big),
\]
a quantity equal to $d$ as soon as $n\ge m_0$, the chain of kernels being stationary. Since $\ran((\Id-B_1)^{n+1})\subseteq\ran((\Id-B_1)^n)$ and these two subspaces have the same finite codimension $d$, they coincide. Hence $\dsc(\Id-B_1)\le m_0<\infty$.

The operator $\Id-B_1$ is therefore Fredholm, with finite ascent and descent: $\Id-B_1\in\Br(\V)$.
\end{proof}

\begin{lemma}[factorization]\label{lem:fact}
Let $A\in\B(\V)$, $K\in\Bo(\V)$ with $AK=KA$, and $q\in\HH$. If $\Qq{A+K}$ is invertible, then $\Qq{A}\in\Br(\V)$.
\end{lemma}

\begin{proof}
Set $B:=\Qq{A}$, $R:=\Qq{A+K}$, assumed invertible, and
\[
K':=2AK+K^2-2\operatorname{Re}(q)K,
\]
a compact operator since $K$ is. Identity~\eqref{eq:pert} of Remark~\ref{rem:pert} reads $R=B+K'$, i.e. $B=R-K'$.

The operator $K'$ commutes with $A$: indeed $[AK,A]=A(KA)-A(AK)=0$, $[K^2,A]=0$ and $[K,A]=0$. It also commutes with $K$: from $AK=KA$ we get $KAK=AK^2$, so $[K',K]=2(AK^2-KAK)=0$. Consequently $K'$ commutes with $A+K$, hence with $R$, which is a polynomial with real coefficients in $A+K$, hence also with $R^{-1}$.

Set $B_1:=R^{-1}K'$, a compact operator. Since $R$ commutes with $K'$ and with $R^{-1}$, it commutes with $B_1$, hence with $\Id-B_1$. Finally
\[
R(\Id-B_1)=R-RR^{-1}K'=R-K'=B.
\]
By Lemma~\ref{lem:invar}, applied to $X=\Id-B_1$ and $R$ invertible commuting with $X$,
\[
\asc(B)=\asc(\Id-B_1),\qquad \dsc(B)=\dsc(\Id-B_1),
\]
quantities that are finite by Lemma~\ref{lem:RS}. Moreover $R$, invertible, and $\Id-B_1$, Browder, are both Fredholm by (R3); by (R2), their product $B=R(\Id-B_1)$ is Fredholm. Thus $B\in\F(\V)$ with $\asc(B),\dsc(B)<\infty$, that is, $B\in\Br(\V)$.
\end{proof}

\begin{corollary}\label{cor:stab}
Let $A\in\B(\V)$ and $K\in\Bo(\V)$ with $AK=KA$. Then $\sigb(A)\subseteq\sS(A+K)$.
\end{corollary}

\begin{proof}
By contraposition: if $q\notin\sS(A+K)$, then $\Qq{A+K}$ is invertible, so $\Qq{A}\in\Br(\V)$ by Lemma~\ref{lem:fact}, that is, $q\notin\sigb(A)$.
\end{proof}

\section{Characterization of the Browder S-spectrum}\label{sec:charac}

\begin{lemma}[perturbation by the Riesz projector]\label{lem:proj}
Let $A\in\B(\V)$ and $q\in\sd(A)$. Then $P_{[q]}$ is compact and commutes with $A$, and $\Qq{A+P_{[q]}}$ is invertible; in particular $q\notin\sS(A+P_{[q]})$.
\end{lemma}

\begin{proof}
Since $q\in\sd(A)$, the sphere $[q]$ is isolated in $\sS(A)$ and $P_{[q]}$ has finite rank, hence is compact by (R2); it commutes with $A$ by (R5). Since $P_{[q]}$ is a projector, Remark~\ref{rem:pert} gives
\[
\Qq{A+P_{[q]}}=\Qq{A}+\big[2A+(1-2\operatorname{Re}(q))\Id\big]P_{[q]},
\]
an operator that is invertible by (R7) (and by Remark~\ref{rem:reel} when $q$ is real).
\end{proof}

\begin{lemma}[localization by nilpotence]\label{lem:loc}
Let $M$ be a nonzero right quaternionic Hilbert space, $T\in\B(M)$ and $q\in\HH$. If $Q_q(T)$ is nilpotent, then $\sS(T)=[q]$.
\end{lemma}

\begin{proof}
\emph{Coprimality.} Let us first show that, for $s\in\HH$, the monic real polynomials $Q_q$ and $Q_s$ of degree $2$ are coprime in $\RR[x]$ as soon as $s\notin[q]$. Suppose they have a nonconstant common divisor $D$. If $\deg D=2$, then $D=Q_q=Q_s$ by monicity, whence $\operatorname{Re}(q)=\operatorname{Re}(s)$ and $|q|=|s|$, that is, $[q]=[s]$. If $\deg D=1$, say $D=x-a$ with $a\in\RR$, then $Q_q$ has the real root $a$; its discriminant $4\operatorname{Re}(q)^2-4|q|^2$, which is always $\le0$, must then be zero, which forces $q\in\RR$, $Q_q=(x-q)^2$ and $a=q$. The same reasoning applied to $Q_s$ gives $s\in\RR$ and $a=s$, whence $q=s$ and $[q]=[s]$. In both cases $s\in[q]$.

\emph{Inclusion $\sS(T)\subseteq[q]$.} Let $s\notin[q]$ and $u,v\in\RR[x]$ be such that $uQ_q+vQ_s=1$ (Bézout identity, legitimate by the above). Set $N:=Q_q(T)$, nilpotent, say $N^k=0$. Since $u(T)$ and $N$ are polynomials with real coefficients in $T$, they commute, so $(u(T)N)^k=u(T)^kN^k=0$: the operator $u(T)N$ is nilpotent. Evaluating the Bézout identity at $T$,
\[
v(T)\,Q_s(T)=\Id_M-u(T)N=:W,
\]
and $W$ is invertible, with inverse $\sum_{j=0}^{k-1}(u(T)N)^j$. Since $v(T)$ and $Q_s(T)$ commute, we also have $Q_s(T)v(T)=W$; hence $W^{-1}v(T)$ is a left inverse and $v(T)W^{-1}$ a right inverse of $Q_s(T)$, which is consequently invertible. Thus $s\notin\sS(T)$.

\emph{Inclusion $[q]\subseteq\sS(T)$.} The operator $Q_q(T)=N$ is nilpotent on a nonzero space, hence not invertible: $q\in\sS(T)$. By axial symmetry (R4), $[q]\subseteq\sS(T)$.
\end{proof}
The following statement is that of Theorem~4.2 of \cite{AJ26}, where it is formulated as the equivalence, for $q\in\sS(A)$, between ``$q$ is a Riesz point of $A$'' and ``$\Qq{A}$ is Fredholm with finite ascent and descent''. The proof below is independent.

\begin{lemma}\label{lem:A}
For every $A\in\B(\V)$, $\ \sigb(A)=\sS(A)\setminus\sd(A)$.
\end{lemma}
\begin{proof}
\textbf{$\sigb(A)\subseteq\sS(A)\setminus\sd(A)$:}
Let us first check $\sigb(A)\subseteq\sS(A)$: if $q\in\rS(A)$, then $\Qq{A}$ is invertible, so $\Qq{A}\in\G(\V)\subset\Br(\V)$ by (R3) and $q\notin\sigb(A)$.

Now let $q\in\sd(A)$. By Lemma~\ref{lem:proj}, $K:=P_{[q]}$ is compact, commutes with $A$, and $\Qq{A+K}$ is invertible. Lemma~\ref{lem:fact} then gives $\Qq{A}\in\Br(\V)$, that is, $q\notin\sigb(A)$.

\medskip
\textbf{$\sS(A)\setminus\sd(A)\subseteq\sigb(A)$:}
Let $q\in\sS(A)\setminus\sd(A)$ and suppose for contradiction that $q\notin\sigb(A)$, that is, $B:=\Qq{A}\in\Br(\V)$. Then $\asc(B)=\dsc(B)=m<\infty$ (\cite[Lemma~7.3(b)]{MT19}). Since $q\in\sS(A)$, the operator $B$ is not invertible, so $m\ge1$; in particular $\ker B\ne\{0\}$, since otherwise $\asc(B)=0$, and hence $\ker(B^m)\supseteq\ker B\ne\{0\}$.

Since $B\in\Br(\V)\subset\F(\V)$ by (R3), (R2) gives $B^m\in\F(\V)$, so $\ran(B^m)$ is closed. By (R1), $\V=\ran(B^m)\oplus\ker(B^m)$, so these two subspaces are closed, and $A$-invariant since $B$ is a real polynomial in $A$, hence commutes with $A$.

\emph{On $\ker(B^m)$.} The operator $Q_q\big(A|_{\ker(B^m)}\big)=B|_{\ker(B^m)}$ is nilpotent, of exponent at most $m$, and $\ker(B^m)\ne\{0\}$; by Lemma~\ref{lem:loc},
\[
\sS\big(A|_{\ker(B^m)}\big)=[q].
\]

\emph{On $\ran(B^m)$.} The operator $B|_{\ran(B^m)}$ is injective: if $x\in\ran(B^m)$ and $Bx=0$, then $x\in\ker B\cap\ran(B^m)\subseteq\ker(B^m)\cap\ran(B^m)=\{0\}$. It is surjective since $\dsc(B)=m$ gives $B(\ran(B^m))=\ran(B^{m+1})=\ran(B^m)$. Since the subspace $\ran(B^m)$ is closed, hence complete, the open mapping theorem ensures that the inverse of this bijection is bounded: $Q_q(A|_{\ran(B^m)})=B|_{\ran(B^m)}$ is invertible in $\B(\ran(B^m))$, so $q\notin\sS(A|_{\ran(B^m)})$. By axial symmetry of this S-spectrum (R4),
\[
[q]\cap\sS\big(A|_{\ran(B^m)}\big)=\emptyset.
\]

\emph{Conclusion.} The projector $P$ of $\V$ onto $\ran(B^m)$ along $\ker(B^m)$ is bounded (the two subspaces are closed and complementary, so the closed graph theorem applies) and commutes with $A$, these subspaces being $A$-invariant. By (R6),
\[
\sS(A)=\sS\big(A|_{\ker(B^m)}\big)\cup\sS\big(A|_{\ran(B^m)}\big)=[q]\;\sqcup\;\sS\big(A|_{\ran(B^m)}\big),
\]
the union being disjoint by the previous paragraph. The complement of $[q]$ in $\sS(A)$ is therefore $\sS(A|_{\ran(B^m)})$, closed by (R4); the sphere $[q]$ is thus open and closed in $\sS(A)$, that is, isolated. This formulation covers the degenerate case $\ran(B^m)=\{0\}$, where $\sS(A|_{\ran(B^m)})=\emptyset$ and $\sS(A)=[q]$.

Finally $B=\Qq{A}\in\Br(\V)\subset\W(\V)$ by (R3). Since the sphere $[q]$ is isolated, (R8) gives $q\in\sd(A)$, contradicting the hypothesis. Hence $q\in\sigb(A)$.
\end{proof}

With Lemma~\ref{lem:A} and Lemma~\ref{lem:proj} in hand, together with Corollary~\ref{cor:stab} from Section~\ref{sec:fact}, we can now establish the main theorem of the note: the characterization of the Browder S-spectrum translates into an intersection formula over the set of compact operators commuting with $A$.

\begin{theorem}\label{thm:main}
For every $A\in\B(\V)$,
\[
\sigb(A)=\bigcap_{K\in\CA}\sS(A+K),\qquad \CA=\{K\in\Bo(\V):AK=KA\}.
\]
\end{theorem}

\begin{proof}
 Corollary~\ref{cor:stab} gives $\sigb(A)\subseteq\sS(A+K)$ for each $K\in\CA$; intersecting over $K\in\CA$, we obtain $\sigb(A)\subseteq\bigcap_{K\in\CA}\sS(A+K)$. To prove the converse let $q\notin\sigb(A)$. By Lemma~\ref{lem:A}, either $q\in\rS(A)$, or $q\in\sd(A)$. In the first case, $K=0\in\CA$ works, since $q\notin\sS(A)$. In the second, Lemma~\ref{lem:proj} provides $K=P_{[q]}\in\CA$ with $q\notin\sS(A+P_{[q]})$. In both cases there exists $K\in\CA$ such that $q\notin\sS(A+K)$, so $q\notin\bigcap_{K\in\CA}\sS(A+K)$.
\end{proof}

The following theorem establishes the stability of the Browder S-spectrum under commuting compact perturbations for arbitrary bounded operators.

\begin{theorem}[pointwise stability]\label{thm:pointwise}
For every $A\in\B(\V)$ and every $K\in\CA$,
\[
\sigb(A+K)=\sigb(A).
\]
\end{theorem}

\begin{proof}
Fix $q\in\HH$ and set $B:=\Qq{A}$. As in the proof of Lemma~\ref{lem:fact}, the operator
\[
K':=2AK+K^2-2\operatorname{Re}(q)K
\]
is compact and commutes with $A$; since $B=Q_q(A)$ is a real polynomial in $A$, it follows that $K'$ commutes with $B$ as well. Identity~\eqref{eq:pert} of Remark~\ref{rem:pert} reads $\Qq{A+K}=B+K'$.

Since $K'$ is compact, (R2) gives $B\in\F(\V)\iff B+K'\in\F(\V)$: apply the stability of the Fredholm property under compact perturbation once to the pair $(B,K')$, and once to the pair $(B+K',-K')$ to recover $B=(B+K')+(-K')$.

Moreover, since $K'$ commutes with $B$, \cite[Thm~3.9]{KD24} gives
\[
\asc(B)<\infty \iff \asc(B+K')<\infty,\qquad \dsc(B)<\infty \iff \dsc(B+K')<\infty.
\]
Indeed, the forward implications follow directly from \cite[Thm~3.9]{KD24}; the reverse implications follow by applying the same theorem to the pair $(B+K',-K')$, noting that $-K'$ is compact and commutes with $B+K'$.

Combining the two equivalences, $B\in\Br(\V)\iff B+K'\in\Br(\V)$, that is, $\Qq{A}\in\Br(\V)\iff\Qq{A+K}\in\Br(\V)$. As $q\in\HH$ was arbitrary, $q\in\sigb(A)\iff q\in\sigb(A+K)$, whence $\sigb(A+K)=\sigb(A)$.
\end{proof}

\begin{corollary}[finite-rank refinement]\label{cor:finite}
For every $A\in\B(\V)$,
\[
\sigb(A)=\bigcap_{K\in\CA\cap\Boo(\V)}\sS(A+K).
\]
\end{corollary}

\begin{proof}
Since $\Boo(\V)\subset\Bo(\V)$ by (R2), $\CA\cap\Boo(\V)\subseteq\CA$; intersecting over the smaller index set $\CA\cap\Boo(\V)$ can only enlarge the intersection, so Theorem~\ref{thm:main} gives
\[
\sigb(A)=\bigcap_{K\in\CA}\sS(A+K)\ \subseteq\ \bigcap_{K\in\CA\cap\Boo(\V)}\sS(A+K).
\]

Conversely, let $q\notin\sigb(A)$. By Lemma~\ref{lem:A}, either $q\in\rS(A)$, in which case $K=0\in\CA\cap\Boo(\V)$ satisfies $q\notin\sS(A+K)$; or $q\in\sd(A)$, in which case the Riesz projector $K=P_{[q]}$ has finite rank by definition of $\sd(A)$ and lies in $\CA\cap\Boo(\V)$ by (R5), and Lemma~\ref{lem:proj} gives $q\notin\sS(A+K)$. In either case there is $K\in\CA\cap\Boo(\V)$ with $q\notin\sS(A+K)$, so $q\notin\bigcap_{K\in\CA\cap\Boo(\V)}\sS(A+K)$.
\end{proof}

\begin{remark}[nature of the two statements]\label{rem:corollaries}
Theorem~\ref{thm:pointwise} and Corollary~\ref{cor:finite} are of a different nature. Corollary~\ref{cor:finite} costs nothing beyond Theorem~\ref{thm:main}: its proof merely revisits the argument already given, observing that the only perturbations ever invoked there ($K=0$ and the Riesz projector $K=P_{[q]}$) are of finite rank. Theorem~\ref{thm:pointwise}, by contrast, is logically independent of both Theorem~\ref{thm:main} and the factorization lemma: it rests instead on the commuting ascent/descent stability theorem of \cite{KD24}, applied pointwise to the pseudo-resolvent at each $q\in\HH$, with the equivalence obtained by the symmetry argument noted in the proof. In exchange for this additional input it yields more than Corollary~\ref{cor:stab}, namely equality in place of the single inclusion $\sigb(A)\subseteq\sS(A+K)$.
\end{remark}

\section{Conclusion}\label{sec:conclusion}
 
The intersection formula $\sigb(A)=\bigcap_{K\in\CA}\sS(A+K)$ is established (Theorem~\ref{thm:main}), resting on the factorization lemma (Lemma~\ref{lem:fact}) and on the characterization $\sigb(A)=\sS(A)\setminus\sd(A)$, for which an independent proof is given (Lemma~\ref{lem:A}). Except for the quaternionic inputs (R5), (R7) and (R8), the argument uses only real polynomial algebra, a compactness/orthogonality argument, and Fredholm theory. The finite-rank refinement of the intersection (Corollary~\ref{cor:finite}) is also obtained, together with, by an independent argument, the stronger pointwise statement $\sigb(A+K)=\sigb(A)$ for fixed $K\in\CA$ (Theorem~\ref{thm:pointwise}).

\end{document}